\documentclass{amsart}
\usepackage{amsthm}
\usepackage{amsmath, amssymb, amsthm, amscd, amsfonts, dsfont, mathrsfs}
\usepackage[utf8]{inputenc}
\usepackage{color}

\newtheorem{theorem}{Theorem}

\newtheorem{definition}{Definition}[section]

\newtheorem{lem}[theorem]{Lemma}

\newtheorem{proposition}[theorem]{Proposition}

\newtheorem{corollary}[theorem]{Corollary}

\newtheorem*{theorem*}{Theorem}

\newtheorem*{corollary*}{Corollary}

\newtheorem*{theoremA}{Theorem A}

\newtheorem*{theoremB}{Theorem B}

\newtheorem*{remark*}{Remark}

\newtheorem*{proposition*}{Proposition}

\newtheorem*{question}{Question}

\title[Random quadratic polynomials]{Total disconnectedness of Julia sets of random quadratic polynomials}

\author{Krzysztof Lech}
\author{Anna Zdunik}
\address{Institute of Mathematics, University of Warsaw,
ul.~Banacha~2, 02-097 Warszawa, Poland}
\thanks{The research  was supported in part by the National Science Centre, Poland, grant no 2018/31/B/ST1/02495.}
\begin{document}
\begin{abstract}
For a sequence of complex parameters $\{ c_n \}$ we consider the compositions of functions $f_{c_n} (z) = z^2 + c_n$, 
which is the non-autonomous version of the classical quadratic dynamical system. 
The definitions of Julia and Fatou sets are naturally generalized to this setting.
We answer a question posed by Brück, B\"uger and Reitz, whether the Julia set for such a sequence is almost always totally disconnected,
if the values $c_n$ are chosen randomly from a large disk. Our proof is easily generalized to answer a lot of other related questions regarding typical connectivity of the random Julia set. In fact we prove the statement for a much larger family of sets than just disks, in particular if one picks $c_n$ randomly from the main cardioid of the Mandelbrot set, then the Julia set is still almost always totally disconnected.
\end{abstract}

\maketitle

\section{Introduction}\label{sec:intro}
We consider non--autonomous compositions of quadratic polynomials $f_c=z^2+c$, where, at each step  $c$ is chosen randomly
from some bounded Borel $V\subset \mathbb C$ (e.g., the disc $\mathbb D(0,R)$).
Let us introduce the parameter space
$\Omega=V^\mathbb N.$
The space $\Omega$ is equipped with a natural left shift map $\sigma$.
Namely, 
for every $\omega\in\Omega$, $\omega=(c_0,c_1, c_2,\dots)$ put
$$\sigma(\omega)=(c_1, c_2, \dots).$$

\noindent Next, for every $\omega\in\Omega$, $\omega=(c_0,c_1,\dots)$ denote by  $f_\omega$ the map $f_{c_0}$.

\noindent Then the  non-autonomous composition  $f^n_\omega$ is given by the formula

$$f^n_\omega:=f_{c_{n-1}}\circ f_{c_{n-2}}\circ \dots \circ f_{c_0}.$$

\noindent The global dynamics can be described as a skew product
$F:\Omega\times \mathbb C\to \Omega\times\mathbb C,$
$$F(\omega, z)=\left (\sigma(\omega), f_\omega(z)\right ).$$

\noindent Then, for all $n\in\mathbb N$, we have that

$$F^n(\omega,z)=\left (\sigma^n(\omega), f_\omega^n(z)\right ).$$

So, every sequence $\omega\in\Omega$ determines a sequence of non-autonomous iterates:
$\left (f_\omega^n\right )_{n\in\mathbb N}.$

Let $\mu$ be a Borel probability measure on $V$. We denote by $\mathbb P$ the product distribution on $\Omega$ generated by $\mu$.
Then $(\Omega,\mathbb P)$ becomes a measurable space, and $\sigma:\Omega\to\Omega$ is an ergodic
measure preserving endomorphism.

\

Analogously to the autonomous case, it is natural to consider the following objects:

\begin{itemize}
\item{} (\emph{escaping set}, or \emph{basin of infinity}) $$\mathcal A_\omega=\{z\in \mathbb C:f^n_\omega(z)\xrightarrow[n\to\infty]{} \infty\}$$
\item{} (non-autonomous Julia set) 
\begin{equation}\label{eq:julia}
J_\omega=\{z\in\mathbb C: \text{for every open set }~  U\ni z~\text{the family}~{f^n_\omega}_{|U}~  \text{is not normal.}\}
\end{equation}
\item{} (non- autonomous filled-in Julia set) 
\begin{equation}\label{eq:filled}
K_\omega=\mathbb C\setminus \mathcal A_\omega.
\end{equation}
\end{itemize}

\

The following proposition, which can be found in \cite{B} (Theorem 1) is analogous to the autonomous case.

\begin{proposition} Let $\omega\in \mathbb D(0,R)^\mathbb N$. Then 
$$J_\omega=\partial\mathcal A_\omega,$$
\end{proposition}

Let us also note the following straightforward observations:

\begin{proposition}
For every $\omega\in\Omega$
\begin{itemize}
\item{} $J_{\sigma\omega}=f_\omega(J_\omega),$
\item{} $\mathcal A_{\sigma\omega}=f_\omega(\mathcal A_{\omega}).$
\end{itemize}
\end{proposition}

\

The study of iterates of non-autonomous and random rational maps, and, in particular, non- autonomous and random polynomials, originated by the seminal 
paper \cite{FS} by J. Fornæss and N. Sibony. It was since developed by many authors.

A systematic study of dynamics on non- autonomous and random dynamics of quadratic polynomials was done by R. Br\"uck, M. B\"uger, S.Reitz, see \cite{BBR, RB, B}. Some other results related to random polynomial dynamics in general have also been achieved by M. Comerford, in \cite{COM}, \cite{COM2}. Finally in \cite{MSU} V. Mayer, M. Skorulski and M. Urbański among other results confirm a conjecture by R. Br\"uck and M. B\"uger, concerning the typical Hausdorff dimension of a certain random quadratic Julia set. 

In \cite{BBR} the authors focus on the question of the connectedness of the Julia set, giving, among other results, a transparent sufficient and necessary condition for the Julia set to be connected:

\begin{theorem*}[Theorem 1.1. in \cite{BBR}]
Let $\omega\in\mathbb D(0,R)^\mathbb N$, $R>0$. The Julia set $J_\omega$ is disconnected if and only if there exists $k\in\mathbb N$ such that
$$f^n_{\sigma^k\omega}(0)\xrightarrow[n\to\infty]{}\infty.$$
\end{theorem*}
 Note that the point $0$ plays a special role, since it is a common critical points of all maps $f_c$.
 Recall that in the autonomous case, i.e, the iterates of a single map  $f_c$,  the Julia set is disconnected
if and only if $f^n_c(0)\xrightarrow[n\to\infty]{}\infty$. Moreover, if the Julia set $J(f_c)$ is disconnected, then it is totally disconnected. The last statement is no longer true in the non-autonomous case considered here; for example, one can easily construct sequences $\omega$ for which $J_\omega$ has finitely many connected components.

\

Looking at the above characterization of connected Julia sets $J_\omega$, one may conjecture that the condition 
$$f^n_{\sigma^k\omega}(0)\xrightarrow[n\to\infty]{}\infty\quad\text{for every}\quad  k\in\mathbb N$$ is the right characterization of totally disconnected Julia sets $J_\omega$.

However, this condition is  neither necessary nor sufficient. Indeed, in \cite{BBR} the authors 
construct  an example of a sequence $\omega\in\mathbb D(0,R)^\mathbb N$ such that for every $k\in\mathbb N$ 
$f^n_{\sigma^k\omega}(0)\to \infty$ as $n\to\infty$, but the Julia set $J_\omega$ is not totally 
disconnected (see Example 4.4 in \cite{BBR}). On the other hand,  Example 4.5 in the same paper shows 
that the Julia set may be totally disconnected even if for infinitely many $k\in \mathbb N$ $f^n_{\sigma^k
\omega}(0)$ does not tend to infinity as $n\to\infty$.

Clearly, the behaviour of the (typical)  dynamics depends on the domain from which the parameters $c_n$ are chosen. In particular in case of a disk $\mathbb D(0, R)$, the dynamics depend on $R$. 

If $R\le 1/4$ then for every $\omega\in \mathbb D^\mathbb N$, $\omega=(c_i)_{i=0}^\infty$, the Julia set $J_\omega$ is connected (see Remark 1.2 in \cite{BBR}). Note that in this case all parameters $c_i$ are chosen from the main cardioid in the Mandelbrot set.

For $R>1/4$ the situation changes drastically. Indeed, the disc $\mathbb D(0,R)$ now contains parameters from the complement of the Mandelbrot set $\mathcal{M}$. So, it is evident that putting, for instance, $\omega=(c,c,c,\dots)$, where $c\in \mathbb {D}(0,R)\setminus\mathcal M$, one obtains a totally disconnected Julia set $J_\omega$.

This motivates the following question, which was raised in  \cite{BBR} and \cite{RB}: what is a typical behaviour of the Julia set $J_\omega$, in terms of connectedness?
 More formally, in   \cite{BBR} and \cite{RB} the authors introduce  subsets of $\Omega=\mathbb D(0,R)^{\mathbb N}$, denoted by  $\mathcal D$, $\mathcal D_N$, $\mathcal D_\infty$, $\mathcal T$, and described in terms of connectedness:
$$\mathcal D=\{\omega\in\Omega: J_\omega\quad\text{is disconnected}\}$$
$$\mathcal D_N=\{\omega\in\Omega: J_\omega\quad\text{has at least $N$ connected components}\}$$ 
$$\mathcal D_\infty=\{\omega\in\Omega: J_\omega\quad\text{has infinitely many connected  components}\}$$
$$\mathcal T=\{\omega\in\Omega: J_\omega\quad\text{is totally disconnected}\}$$
$$\mathcal F=\{\omega\in\Omega: \forall k\in\mathbb N ~  f^n_{\sigma^k\omega}(0)\xrightarrow[n\to\infty]{}\infty\}$$
Clearly, $\mathcal D\supset \mathcal D_N\supset \mathcal D_\infty\supset \mathcal T$. But, as mentioned above, the set $\mathcal F$ is neither contained in nor it contains $\mathcal T$.

Here, typicality may be understood in topological or metric sense. The space $\Omega=\mathbb D(0,R)^\mathbb N$ carries the natural product topology
induced by the standard topology on $\mathbb D(0,R)$.
Note that this topology is completely metrizable.

The space $\Omega$ also carries the natural product measure $\mathbb P:=\otimes_{n=0}^\infty \lambda_R$
where each $\lambda_R$ is the normalized Lebesgue measure on $\mathbb D(0,R)$.
In \cite{BBR} the authors prove that $\mathbb P(\mathcal D)=1$ (Theorem 2.3 in \cite{BBR}). 
It can be  deduced from the proof, in a rather straightforward way, that $\mathbb P(\mathcal F)=1$ and also (although it is 
not explicitly stated in the paper) that $\mathbb P(\mathcal D_\infty)=1$.

The work \cite{RB} deals with topological aspects of typicality of the above sets.
In particular, the author proves  (assuming $R>1/4$) that

\begin{itemize}
\item{} the set $\mathcal T$ is dense in $\Omega$ (Theorem 1.1 )
\item{} the set $\mathcal D_\infty$ has empty interior in $\Omega$ (Theorem 1.2)
\item{} for every $N\in\mathbb N$ the set $\mathcal D_N$ is an open dense subset of $\Omega$, which immediately implies that
\item{} the set $\mathcal D_\infty$ is of the second Baire category.
\end{itemize}
In \cite{RB} the author asked if the set $\mathcal T$  is also of the second Baire category. This question was positively answered by Z. Gong, W. Qiu and Y. Li in \cite{QIU}. 

\

However, the question about metric typicality of $\mathcal T$,  formulated in \cite{BBR}  remained open until now:

\begin{question}{{\rm [BBR]}}
Is it true that $\mathbb P(\mathcal T)=1 $ provided that $R>1/4$ is large enough?
\end{question}

In this paper, we answer the above  question  positively, providing, moreover a number of  stronger statements. Precisely, we prove the following.

\begin{theoremA}\label{thm:A}
Let $R>1/4$. Consider $\Omega=\mathbb D(0,R)^\mathbb N$ equipped with the product distribution $\mathbb P:=\otimes_{n=0}^\infty \lambda_R$.  Let 
$$\mathcal T=\{\omega\in\Omega: J_\omega\quad\text{is totally disconnected}\}$$
Then $\mathbb P (\mathcal T)=1$.
\end{theoremA}
In other words, a typical (metrically) Julia set $J_\omega$ is totally disconnected.

One might expect that the phenomenon described in Theorem A is based on the fact that for $R>1/4$ the disc $\mathbb D(0,R)$ 
intersects the complement of the Mandelbrot set $\mathcal M$. However, the following generalization shows that the analogous statement holds true also for domains which 
are completely contained in the Mandelbrot set. Namely, we have the following generalization of Theorem A.

\begin{theoremB}
Let $V$ be an open and bounded set such that $\mathbb{D}(0,\frac{1}{4}) \subset V$
and $V \neq \mathbb{D}(0,\frac{1}{4})$. Consider the space $\Omega = V^{\mathbb{N}}$ equipped with the  product $\mathbb P$  of uniform distributions on $V$. 
Then for $\mathbb P$--almost every sequence $\omega\in\Omega$ the Julia set $J_\omega$ is totally disconnected.
\end{theoremB}
Theorem B leads immediately to the following corollary.

\begin{corollary*}[Corollary~\ref{thm:cardioid}]
 Let $\Omega = B^{\mathbb{N}}$ where $B$ is the main cardioid of the Mandelbrot set, 
 and let $\Omega$ be equipped with the product of uniform distributions on $B$. Then for almost every sequence $\omega \in \Omega$ 
 the Julia set $J_\omega$ is totally disconnected.
\end{corollary*}

Moreover, a number of applications of our  approach, possible generalization  and further results are presented in Section \ref{sec:remarks}.  

\section{Green's function.}

\noindent {\bf Notation.} For every $r>0$ denote $\mathbb D_r:=\mathbb D(0,r)$ and
$\mathbb D_{r}^* := \mathbb{C} \setminus \overline{\mathbb D}_{r}$.

\noindent We write $\omega=(c_0,c_1\dots)$ in various contexts to denote an infinite sequence of parameters, even if no probability distribution is specified.
For such a sequence we use both notations:

$$f^n_\omega=f_{c_{n-1},c_{n-2},\dots,c_1,\dots c_0}=f_{c_{n-1}}\circ f_{c_{n-2}}\circ \dots \circ f_{c_1}\circ f_{c_0}$$ 

\

\subsection{Green's function on $\mathcal A_\omega$.}
We recall the proposition proved in  \cite{FS}, which we state in a slightly different form.

\begin{proposition}
Let $V$ be a bounded Borel subset of $\mathbb C$, put $\Omega=V^\mathbb N$. Let $\mu$ be a Borel probability measure on $V$, 
and $\mathbb P$ - the product distribution on $\Omega$ generated by $\mu$. 
For every $\omega\in\Omega$ the following limit exists: $g_\omega:\mathcal A_\omega\to \mathbb R$:

\begin{equation}\label{eq:green}
g_\omega(z)= \lim_{n\to\infty} \frac{1}{2^n}\log |f^n_\omega(z)|.
\end{equation}
 The function $z\mapsto g_\omega(z)$ is the Green's function on $\mathcal A_\omega$ with pole at infinity. Putting $g_\omega\equiv 0$ on the complement of $\mathcal A_\infty$, $g_\omega$ extends continuously to the whole plane.  With $z$ fixed, the function $\omega\mapsto g_\omega(z)$ is $\mathbb P$-- measurable.
 \end{proposition}

 This is a generalization of a well- known formula for the autonomous case: for the map $f_c(z):=z^2+c$ and its  basin of infinity $\mathcal A_c$, the Green's function with a pole at infinity is given by:
 
$$g_c(z)= \lim_{n\to\infty} \frac{1}{2^n}\log |f^n_c(z)|.$$

\begin{corollary}
We have 
\begin{equation}\label{eq:invariance}
g_{\sigma\omega}(f_\omega(z))=2 g_\omega(z).
\end{equation}
\end{corollary}
\begin{proof}
This follows directly from the formula \eqref{eq:green}, defining the Green's function $g_\omega$.
\end{proof}

{\bf {Observation.}} Critical points of $g_\omega$. Writing $f_i$ for $f_{c_i}$, we see that $g_\omega$ has critical points at each point of the following sets:
$$\mathcal C_0=\{0\}$$
$$\mathcal C_1=f_1^{-1}(0)$$
$$\mathcal C_2=f_1^{-1}f_2^{-1}(0)\ldots$$
$$\mathcal C_k=f_1^{-1}f_2^{-1} f_3^{-1}\dots f_k^{-1}(0)\ldots$$
Let us note that in the autonomous case the critical points of $g_c$ form a ''tree'', i.e.,  
$\mathcal C_k=f^{-1}(\mathcal C_{k-1}),$
while in a general non--autonomous case the set $\mathcal C_k$ is not a preimage of $\mathcal C_{k-1}$ under 
any the maps $f_i$. 

\

\subsection{Estimates for Green's function.}\label{sec:prep}

\begin{proposition} \label{prop:log}
For every $\varepsilon>0$, $R>0$  there exists $R_0>0$ such that for every $\omega\in\mathbb D(0,R)^{\mathbb N}$ we have that
\begin{equation}\label{eq:1}
f_\omega(\mathbb D_{R_0}^*)\subset \mathbb D^*_{2R_0},
\end{equation}
\begin{equation}\label{eq:2}
\mathbb D_{R_0}^*\subset \mathcal A_\omega,
\end{equation}
\begin{equation}\label{eq:3}
|g_\omega(z)-\log|z||<\varepsilon\quad\text{in}\quad\mathbb D_{R_0}^*.
\end{equation}
\end{proposition}

\begin{proof}
First, since $|c_n|<R$ for all $c$,  one  can  choose $R_1>0$  to ensure 
\begin{equation}\label{eq:escaping}
f_\omega(\mathbb D_{R_0}^*)\subset \mathbb D^*_{2R_0}
\end{equation}
for every $R_0\ge R_1$.
This guarantees \eqref{eq:1} and\eqref{eq:2}.

Let $a_n(z) = \frac{1}{2^n} \log |f^n_{\omega}(z)|$, then we have :

$$
\aligned
a_{n} (z) &= \frac{1}{2^n}(\log |f^{n}_{\omega} (z)|) = \frac{1}{2^n} (\log |z^{2^n}| + \log |1 + \sum^{2^n - 1}_{k = 0} \frac{b_k}{z^{2^n-k}}|)\\
& = \log |z| + \frac{1}{2^n} \log |1 +  \sum^{2^n - 1}_{k = 0} \frac{b_k}{z^{2^n-k}}|
\endaligned
$$
where $b_k$ are some polynomials of variables $c_1,c_2, ... , c_n$. Let us set a fixed $N$. Then, since  $ |c_k| < R$ for all $k$, 
we can pick $R_0\ge R_1$ large enough so that on $\mathbb D_{R_0}^*$ we get
$$|a_N(z) - \log|z|| < \frac{\varepsilon}{2}.$$
On the other hand we have: 

\begin{equation*}
\begin{split}
a_{n+1}(z) & = \frac{1}{2^{n+1}}(\log |f^{n+1}_{\omega} (z)|) \\
& = \frac{1}{2} (\frac{1}{2^n} \log |(f^n_{\omega} (z))^2+c_{n+1}|) \\
& = \frac{1}{2^n} (\log|f^n_{\omega} (z)| + \frac{1}{2} \log |1 + \frac{c_{n+1}}{(f^n_{\omega}(z))^2}|) \\
& = a_n(z) + \frac{1}{2^{n+1}} (\log |1 + \frac{c_{n+1}}{(f^n_{\omega}(z))^2}|)
\end{split}
\end{equation*}

Thus for a large enough $R_0\ge R_1$ we get 
$$|a_N(z) - g_{\omega} (z)| < \frac{\varepsilon}{2}$$
on $\mathbb D_{R_0}^*$. This along with the previous observation yields the desired inequality
$$|g_{\omega}(z) - \log|z|| < \varepsilon$$
which concludes the proof.
\end{proof}
The following is an immediate consequence of item (a) of Proposition~\ref{prop:log}.
\begin{corollary}
For every $\omega\in\mathbb D(0,R)^\mathbb N$ $K_\omega\subset \mathbb D_{R_0}$. 
\end{corollary}

\noindent{\bf Determining constants.}
Now, for every $R$ we 
 fix some  $R_0>R$ satisfying the conditions formulated in Proposition ~\ref{prop:log} with $\varepsilon:=1$, in particular,
 \begin{equation}\label{eq:set_R_0}
 |g_\omega(z)-\log|z||<1\quad\text{in}\quad \mathbb D_{R_0}^*
 \end{equation}
Next, for every $R>0$ let us fix also some $\tilde R_0\in (R_0, R_0^2-R)$, say, \\ $\tilde R_0=\frac{1}{2}(R_0+R_0^2-R)$.
Then for every $\omega\in\mathbb D(0,R)^\mathbb N$, $f_\omega^{-1}(\mathbb D_{\tilde R_0})\subset \mathbb D_{R_0}$.
By Proposition~\ref{prop:log}, 

\begin{equation}\label{eq:set_G}
G=G(R):=\sup_{|R_0|\le |z|\le \tilde R_0}(g_\omega(z))<\infty
\end{equation}

\begin{proposition} For every $R>0$, for every $\omega\in\mathbb D(0,R)^\mathbb N$, 
$$\sup_{z\in\mathbb D_{R_0}}g_\omega(z)\le\log R_0+1.$$
In particular, 
the function $\mathbb D(0,R)^\mathbb N\ni\omega\mapsto  g_\omega(0)$ is bounded above, i.e.,
$$\sup_{\omega\in\mathbb D(0,R)^\mathbb N}g_\omega(0)\le\log R_0+1.$$
\end{proposition}
\begin{proof}
Since $|g_\omega(z)|\le \log|z|+1$ in $\mathbb D^*_{R_0}$, we have, in particular, ${g_\omega}_{|\partial \mathbb D_{R_0}}\le \log R_0+1$. 
By the Maximum Principle, the same estimate holds in the whole disc $\mathbb D_{R_0}$, 
in particular, for $z=0$.  
So, $\sup_{\omega\in \mathbb D(0,R)^\mathbb N} g_\omega(0)\le \log R_0+1$.
\end{proof}
\

\subsection{Escape rate of the critical point.}
We introduce the following definition.
\begin{definition}\label{def:escape}
Let $\omega\in\mathbb D(0,R)^\mathbb N$. For every $z\in\mathbb D(0,R_0)$ we denote by $k(z,\omega)$ the escape time of $z$ from $\mathbb D_{R_0}$:
\begin{equation}\label{eq:escape}
k(z,\omega)=\begin{cases}
\min\{j:| f^j_\omega(z)|\ge R_0\}\quad\text{if} \quad z\in \mathcal A_\omega\\
\infty\quad\text{if}\quad z\in K_\omega.
\end{cases}
\end{equation}
\end{definition}
\begin{proposition}
 For every $z\in\mathcal A_\omega \cap \mathbb D_{R_0}$ 
 \begin{equation}\label{eq:est_green}
(\log R_0-1)2^{-k(z,\omega)}\le g_\omega(z)\le 2(\log R_0+1)  2^{-k(z,\omega)}
\end{equation}
 \end{proposition}
\begin{proof}
 Recall that, by \eqref{eq:invariance}, 

$$g_\omega(z)=g_{\sigma^{k(z,\omega)}\omega}(f^{k(z,\omega)}_\omega(z))\cdot 2^{-k(z,\omega)},$$
and 
$$g_{\sigma^{k(z,\omega)}\omega}(f_\omega^{k(z,\omega)}(z))=2g_{\sigma^{k(z,\omega)-1}\omega}
(f_\omega^{{k(z,\omega)-1}}(z))\le 2(\log R_0+1),$$
since $f_{\omega}^{k(z,\omega) - 1}(z))\in \mathbb D_{R_0}$.

On the other hand,  
$$g_{\sigma^{k(z,\omega)}\omega}(z)\ge \log |f_{\sigma^{k(z,\omega)}\omega}(z)|-1\ge\log R_0-1$$
This implies that \eqref{eq:est_green} holds.
\end{proof}
 Our estimates show that the distribution of the random variable $\log^-g_\omega(0)$
is roughly the same as that of $k(0,\omega)$.

We introduce the following definition.

\begin{definition}\label{def: fast_escaping}
Let $V$ be a bounded Borel subset of $\mathbb C$, $V\subset \mathbb D(0,R)$. 
Let $\mu$ be a probability Borel measure on $V$, and let $\mathbb P$ be the product distribution on $V^\mathbb N$ generated by $\mu$.
Fix the values $R_0=R_0(R)$ and $G=G(R)$ according to \eqref{eq:set_G}.
We say that the critical point is \emph{typically fast escaping} if there  exists $\gamma>0$ such that 
\begin{equation}\label{eq: fast_escaping}
\mathbb P\left(\{\omega\in\Omega: g_\omega(0)<\frac{G}{2^k}\}\right )<e^{-\gamma k}
\end{equation}
\end{definition}

\section{Sufficient condition for total disconnectedness.} 
Recall that in Section~\ref{sec:prep} we assigned, for every $R>0$ the values   $R_0$ and $\tilde R_0$. 
\begin{lem}\label{lem:inter} 
Choose an arbitrary   radius $\rho\in  [R_0, \tilde R_0]$
and let   $D:=\mathbb D_\rho$. Then 
 the filled -in Julia set $K_\omega$, i.e. the set  of points $z$  whose trajectories $f^n_\omega(z)$ do not escape to $\infty$ can be  written as
 
 $$K_\omega := \bigcap_{k\in\mathbb N} (f^k_\omega)^{-1} (D).$$   
\end{lem}
\begin{proof}
Since the trajectory of every point $z\in\bigcap_{k\in\mathbb N} (f^k_\omega)^{-1} (D)$
is bounded, it is clear that $$\bigcap_{k\in\mathbb N} (f^k_\omega)^{-1} (D)\subset K_\omega.$$ On the other hand, if $z\notin \bigcap_{k\in\mathbb N} (f^k_\omega)^{-1} (D)$ then, for some $k\in \mathbb N$, $|f^k_\omega(z)|\ge\rho\ge R_0$, and it follows from the choice of $R_0$ that

$$f^n_\omega(z)=f^{n-k}_{\sigma^k\omega}(f^k_\omega(z))\xrightarrow[n\to\infty]{}\infty,$$
so $z\notin K_\omega$.
\end{proof}

Observe that $\bigcap_{k\in\mathbb N} (f^k_\omega)^{-1} (D)$ is an intersection of a descending sequence of sets. At each level $k$, the set 
$$D^k(\omega):=(f^k_\omega)^{-1} (D)$$
is   a union of pairwise disjoint 
topological discs $D_j^k(\omega)$, each of them being mapped by $f^k_\omega$ onto $D$ with some degree $d_j^k\le 2^k$.

\

Now put $\rho=\tilde R_0$, i.e.,  put $D:=\mathbb D_{\tilde R_0}$.
The following proposition formulates, in terms of  degree of the maps $f^k_\omega:D_j^k(\omega)\to D$,
a sufficient condition for total disconnectedness of the Julia set $J_\omega$.

\begin{proposition}\label{prop:sufficient}

Let $\omega\in\mathbb D(0,R)^\mathbb N$. If  there exists $N\in\mathbb N$ such that for  infinitely many indices $k\in\mathbb N$ , 
for  each component $D^k_j(\omega)$ of the set $D^k(\omega)=(f^k_\omega)^{-1}(D)$ the degree of  the map

$$f^k_\omega:D^k_j(\omega)\to D$$
is at most $N$, 
then the Julia set $J_\omega$ is totally disconnected.  

\end{proposition} 

\begin{proof}

In what follows, to simplify the notation we write $D^k_j$ and $D^k$  in place of $D^k_j(\omega)$ and $D^k(\omega)$, respectively.
Recall that  $R_0$ and $\tilde R_0$ were  chosen in Section~\ref{sec:prep}
in such a way that
\begin{equation}\label{eq:basin}
\forall_{ \nu\in\mathbb D(0,R)^\mathbb N}\quad  \mathbb D_{R_0}^* \subset A_{\omega} \quad\text{and}\quad f_\nu^{-1}(\mathbb D_{\tilde R_0})\subset \mathbb D_{R_0}.
\end{equation}

Denote by $P$ the annulus
$$
P = \{z: R_0 < |z| < \tilde R_0 \}.
$$
For every $k\in\mathbb N $ and for every component $D^k_j$ of $D^k$ the map $f^k_\omega:D^k_j\to D$ is a proper holomorphic map onto $D$.

 By the assumption there exists an increasing sequence of positive integers $\{k_n \}$ such that  the maps

$$
f^{k_n}_{\omega} : D^{k_n}_j \rightarrow D
$$
have degree at most $N$ for all $j$. 

Now, let us divide the annulus $P$
into $N+1$  nested geometric annuli with the same modulus  $M$. 
These $N+1$ annuli all lie in the intersection of $D$ and all basins of infinity $\mathcal A_\nu$, $\nu\in\mathbb D(0,R)^\mathbb N$, by \eqref{eq:basin}.

\

Let us pick a point $z$ in the Julia set, and let $D^{k_n}_{j_n}$ be the component of $D^{k_n}$ such that $z \in D^{k_n}_{j_n}$. \\
 Since the degree of $f^{k_n}_{\omega}$ on $D^{k_n}_{j_n}$ is at most $N$, one of the $N+1$ annuli contains no critical values of $f^{k_n}_{\omega}$; let us  choose such an annulus and denote it by $P_{n}$. Consider now the (possibly smaller) disc $D'\subset D$, bounded by the outer boundary circle of the annulus   $P_{n}$, and let ${D'}^{k_n}_{j_n}$ be the connected component  of $(f^{k_n}_\omega)^{-1}(D')$, containing the point $z$.
The map $f^{k_n}_\omega: {D'}^{k_n}_{j_n}\to D'$ is also proper,  and the preimage  of the annulus $P_{n}$  under this map, denoted here by  $P'_{n}$ is again a (topological)  annulus. The map $f^{k_n}_\omega$ restricted to $P_{n}$ is a covering map, of degree at most $N$, so, the modulus of $P'_{n}$ is at least $M/N$.

The point $z$ lies in some  connected component of $(f^{k_n}_\omega)^{-1}(\mathbb D_{R_0})$ contained in  $D^{k_n}_{j_n}$, so, in particular, it lies in the bounded component of the complement of the annulus $P'_{n}$.

Now, let us recall that, according to the choice of $R_0$ and $\tilde R_0$,   for every\\ $\nu\in\mathbb D(0,R)^\mathbb N$ we have that 
$$f_\nu^{-1}(\mathbb D_{\tilde R_0})\subset  \mathbb D_{R_0}.$$
So, in particular, for every $k\ge 1$,   $f^{-1}_{\sigma^{k-1}\omega}(\mathbb D_{\tilde R_0})
\subset \mathbb D_{R_0}$.
This also implies that for any $k$ and any $\omega\in\mathbb D(0,R)^\mathbb N$,
each component of  $(f^{(k+1)}_\omega)^{-1}(\mathbb D_{\tilde R_0})$ is contained in some component of 
$(f^{k}_\omega)^{-1}(\mathbb D_{R_0})$ (since  each such component is mapped by $f^k_\omega$ onto some component of $f^{-1}_{\sigma^{k-1}\omega}(\mathbb D_{\tilde R_0})$).  Clearly, the same is true with $ k+1$ being replaced by any arbitrary integer   $m>k$.

We shall apply now the above observation for $k:=k_n$ and $m:=k_{n+1}$.
So,  again, for $k_{n+1}$ we find  a topological  annulus $P'_{n+1}$ of modulus at least $M/N$,  in the 
connected component of $(f^{k_{n+1}}_\omega)^{-1}(\mathbb D_{\tilde R_0})$ containing the point $z$, and such that $z$ lies in  the bounded component of the complement of the annulus $P'_{n+1}$.

Using the  above observation we  conclude that the annulus $P'_{n+1}$ is contained in the component of
$(f^{k_n}_\omega)^{-1}(\mathbb D_{R_0})$ containing the point $z$; in particular, it is contained in the bounded component of the complement of $P_n'$.

\
In this way, we obtain a nested infinite sequence of disjoint annuli  $P'_n$, all contained in $\mathcal A_\omega$, the point $z$ being in the bounded component of the complement of each of them.

Now let us fix $n$ and consider the topological annulus $\mathcal P_n$  that is bounded by the 
boundaries of $D$ and $D^{k_n}_{j_n}$. Since it contains the nested sequence of annuli  $P'_{1}, P'_{2}, ..., 
P'_{n}$, each of modulus at least $\frac{M}{N}$,  then, by Gr\"otzsch inequality, it must have modulus 
at least $n \frac{M}{N}$ (see, e.g., \cite{BH}, Proposition 5.4  or \cite{McMullen}, Theorem B5).  This in turn means it 
contains an actual geometric annulus of modulus at least $ n \frac{d}{N}-C$ (where $C$ is some 
constant), which separates the components of the boundary of $\mathcal P_n$.  (see, e.g.,  Theorem 2.1 
in\cite{McMullen}).   Since for every $n$  the connected component of $K_\omega$, containing the point $z$, is contained in the bounded component of the complement of $\mathcal P_n$, this implies that the component of $K_\omega$
containing $z$ must have arbitrarily small diameter, i.e. it is the single point $z$.

Since the choice of the 
point $z$ was  arbitrary, finally this means the Julia set is totally disconnected, which concludes the 
proof of Proposition~\ref{prop:sufficient}.
\end{proof}

\section{Typically fast escaping critical point and total disconnectedness}\label{sec:typically}
In this section, we check that the condition formulated in Definiton~\ref{def: fast_escaping} is 
sufficient to prove that the assumptions of  Propositon~\ref{prop:sufficient} are satisfied for $
\mathbb P$--a.e. $\omega$. More precisely, we prove the following.

\begin{theorem}\label{thm:main_step}
Let $V$ be a bounded Borel subset of $\mathbb C$, $V\subset \mathbb D(0,R)$. Let $R_0$, $G$ be the 
values assigned to $R$ as in Section~\ref{sec:prep}.
Let $\mu$ be a Borel probability measure on $V$ and let $\mathbb P$ be the product distribution on $\Omega= V^
\mathbb N$, generated by $\mu$. 

If the critical point $0$ is \emph{typically fast escaping}, i.e., if 
\eqref{eq: fast_escaping} holds, then the assumptions of  Propositon~\ref{prop:sufficient} are 
satisfied for $\mathbb P$--almost every $\omega\in\Omega$.
Thus, for $\mathbb P$-a.e. $\omega\in\Omega$ the Julia set $J_\omega$ is totally disconnected.
\end{theorem}

Actually, the property from \eqref{eq: fast_escaping} is stronger than necessary, since to apply our proof all that is needed is for the series of probabilities to be convergent. In all our applications the bounds are indeed exponential, nevertheless the reader will soon see that the following remark is also true.

\begin{remark*}
The statement of Theorem \ref{thm:main_step} is still true if one replaces \eqref{eq: fast_escaping} with 
$$
\sum\limits_{k =0 }^{\infty} \mathbb P\left(\{\omega\in\Omega: g_\omega(0)<\frac{G}{2^k}\}\right ) < \infty.
$$
\end{remark*}

Define the sets
$$A_k=\{\omega\in\Omega:g_\omega(0)<\frac{1}{2^k}G\}.$$

Before proving Theorem~\ref{thm:main_step} we explain in the  next proposition  the role of the sets $A_k$  in possible application of Proposition~\ref{prop:sufficient}.
We apply the setting and the notation of Theorem~\ref{thm:main_step}.
\begin{proposition}\label{prop:A_k}
(a):
If 
$$\sigma^i\omega\notin A_{k-i}\quad\text{for all}\quad i=0,\dots,k-1$$

then for every connected component  $D^k_j$ of the preimage $(f_\omega^{k})^{-1}(D)$ the degree of the map

$$f^k_\omega:D^k_j\to D$$ is equal to $1$.

(b):
If the above holds for all but $l$ indices then for every connected component of the set $f^{-k}_\omega(D)$  the degree  of 

$$f^k_\omega:D^k_j\to D$$ is bounded above by $N=2^l$.

Here, $D$ is the disc introduced in Lemma~\ref{lem:inter}.
 
\end{proposition}

\begin{proof}
It follows from \eqref{eq:set_G} that for every $\nu\in\Omega$ and for every $z\in D$  $g_\nu(z)\le G$. 
Let $D^k_*$ be some component of $(f_\omega^{k})^{-1}(D)$. Consider the sequence of maps

$$D^k_*\xrightarrow[f_\omega]{}D^{k-1}_*\xrightarrow[f_{\sigma\omega}]{}D^{k-2}_*{}\dots \xrightarrow[f_{\sigma^{k-2}\omega}]{}D^{1}_*\xrightarrow[f_{\sigma^{k-1}\omega}]{}D, $$
where we denoted by $D^{k-i}_*$ the consecutive images of $D^k_*$ under the maps $f_\omega, f_{\sigma\omega}\dots f_{\sigma^{k-1}\omega}$. 
Note that $f^k_\omega:D^k_*\to D$ is just the composition of the above sequence of maps.
If $D^{k-i}_*$ contains the critical point $0$ then $f_{\sigma^{k-i}}:D^{k-i}_*\to D^{k-i-1}_*$ is a degree two map; otherwise it is univalent.

Now, if 
\begin{equation}\label{eq:univ}
\sigma^{i}\omega\notin A_{k-i}
\end{equation}
then $g_{\sigma^i\omega}(0)>\frac{1}{2^{k-i}}G$, while for every $z\in D^{k-i}_*$ we have that 
$$g_{\sigma^i\omega}(z)=\frac{1}{2^{k-i}}g(f_{\sigma^i\omega}^{k-i}(z))<\frac{1}{2^{k-i}}\cdot G$$
This implies that $0\notin  D^{k-i}_*$ and, consequently, the map $f_{\sigma^{i+1}\omega}:D^{k-i}_*\to 
D^{k-i-1}_*$ is univalent.  So, if \eqref{eq:univ} happens for all $i=0,\dots k-1$ then the map

$$f^k_\omega:D^k_*\to D$$
is univalent, so of degree one.

If \eqref{eq:univ} fails to hold for $l$ indices $i$, then for these indices  the degree of  the map  $f_{\sigma^{i+1}\omega}:D^{k-i}_*\to 
D^{k-i-1}_*$ is equal to one or two, while for all other indices it is equal to one, so that the degree of the composition $f^k_\omega:D^k_*\to D$ is at most $N=2^l$.
Propositon ~\ref{prop:A_k} is proved.
\end{proof}

\

\begin{proof}[Proof of Theorem~\ref{thm:main_step}]
We consider now the extended probability space 

$$\tilde\Omega:=V^\mathbb Z,$$
with product probability, which we denote by $\tilde {\mathbb P}$.
The left shift $\sigma$, considered in $\tilde\Omega$ is now a measurable automorphism of the space $\tilde\Omega$. There is a natural measurable projection

$$\pi:(\tilde\Omega, \tilde{\mathbb P})\to(\Omega,\mathbb P)$$

$$\tilde\Omega\ni(\dots c_{-2},c_{-1},c_0,c_1,c_2,\dots)\xrightarrow[]{\pi}(c_0,c_1,c_2\dots)\in\Omega$$
This projection transforms the measure $\tilde{\mathbb P}$ onto the measure $\mathbb P$, i.e., $
\tilde{\mathbb P}\circ \pi^{-1}=\mathbb P$.

\noindent For each $\tilde\omega\in\tilde\Omega$ the iterates $f^n_\omega$ are defined as previously, i.e.,  for $\tilde\omega=(\dots c_{-2}, c_{-1},c_0,c_1,c_2\dots)$
$$f^n_{\tilde\omega}(z)=f_{c_{n-1}}\circ\dots \circ f_{c_1}\circ f_{c_0}(z).$$
The Julia set is  defined analogously to \eqref{eq:julia} and denoted by $J_{\tilde\omega}$.
Similarly, the Green function $g_{\tilde\omega}$ is defined as in \eqref{eq:green}.

Considering the extended space $\tilde\Omega$ in this context may seem artificial, since  the iterates $f^n_{\tilde\omega}$ depend only on the ''future'', i.e.,
only non- negative items $(c_j)_{j\ge 0}$  are used to define  $f^n_{\tilde\omega}$ or its Julia set. 
Nevertheless, the proof is based on the construction of appropriate backward trajectories, which we shall describe below.
\
Let
$$E_k=\{\tilde\omega\in\tilde\Omega: g_{\sigma^{-k}{\tilde\omega}}(0)\le \frac{1}{2^k}G\}, \quad k=0,1,2,\dots$$

Let us note that the following estimate holds.

\begin{proposition}\label{prop:est_Ek}
If the critical point is typically fast escaping, i.e., if \eqref{eq: fast_escaping} holds, then
$$\tilde{\mathbb P}(E_k)<e^{-\gamma k},$$
where $\gamma$ comes from the estimate formulated in \eqref{eq: fast_escaping}.
\end{proposition}
\begin{proof}
We have the estimates for the measure $\mathbb P$ of the set  $A_k\subset \Omega$, given by \eqref{eq: fast_escaping}. 

Now, let

$$\tilde A_k:=\pi^{-1}(A_k)=V^{\mathbb N}\times A_k$$

Then, 
$$\tilde{\mathbb P}(\tilde A_k)=\mathbb P(A_k).$$

Now, note that  $E_k=\sigma^{k}(\tilde A_k)$, which implies that $$\tilde{\mathbb P}(E_k)=\tilde{\mathbb P}(\tilde A_k)=\mathbb P(A_k)<e^{-\gamma k}.$$
\end{proof}

It follows from Proposition~\ref{prop:est_Ek} and  Borel--Cantelli Lemma that almost every $\tilde \omega\in\tilde\Omega$ belongs to finitely many sets $E_k$.
This implies that there exists $K\in\mathbb N$ and a set $E\subset \tilde\Omega$ such that

$$\tilde{\mathbb P}(E)>0$$
and
$$E\cap (\bigcup_{k=K}^\infty E_k)=\emptyset.$$

Thus, for every $\tilde\omega\in E$ and every  $k\ge K$ he have that

$$g_{\sigma^{-k}\tilde\omega}(0)>\frac{1}{2^k}G$$
Applying Propositon ~\ref{prop:A_k} 
we obtain immediately the following. 
\begin{corollary}\label{cor:degree}
Let  $\tilde \omega\in E$. Then there exists $N:=2^K$ such that  for every $k\in\mathbb N$, and every connected component $D_j^k(\sigma^{-k}(\tilde\omega))$ the degree of the map

$$f^k_{\sigma^{-k}\omega}:D_j^k (\sigma^{-k}(\tilde\omega))\to D$$ is bounded above by $N$.

\end{corollary}

Now, using ergodicity of the left shift $\sigma$ on $\tilde \Omega$ , we conclude that $\tilde{\mathbb P}$--- almost surely a sequence   $\tilde\nu\in\tilde\Omega$
visits $E$ infinitely many times under the iterates of 
$\sigma$.

\

 Let $k\in\mathbb N$. For $\nu\in\Omega$ we introduce the following.

\

\noindent {\bf Property (K,k):} $\sigma^i\nu\in A_{k-i}$ for more than $K$ indices $i\in\{0,\dots k-1\}$

\begin{lem}\label{lem:*} If Property (K,k) holds for $\nu\in\Omega$ and  $\tilde\nu\in\pi^{-1}(\nu)$,
 then $\sigma^k\tilde\nu\notin E$
\end{lem}

\begin{proof}
Indeed, let $\tilde\nu\in\pi^{-1}(\nu)$. Now,  $\sigma^i\nu\in A_{k-i}$ means that 

$$g_{\sigma^i\tilde\nu}(0)=g_{\sigma^i\nu}(0)<\frac{1}{2^{k-i}}G.$$
Putting $m:=k-i$, this can be rewritten as

$$g_{\sigma^{-m}(\sigma^k\tilde\nu)}(0)<\frac{1}{2^m}G.$$
i.e., 
\begin{equation}\label{eq:more}
\sigma^k\tilde\nu\in E_m
\end{equation}

Since \eqref{eq:more} happens for more than $K$ indices $m$, the definition of the set $E$ implies that $\sigma^k\tilde\nu\notin E$.
\end{proof}

\

Let $B$ be the set of elements  $\nu\in\Omega$, for which {\bf Property (K,k)} happens for all but finitely many indices $k$. Put $\tilde B:=\pi^{-1}(B)$. It follows from Lemma ~\ref{lem:*}  that every point $\tilde\nu\in\pi^{-1}(B)$ visits $E$ finitely many times under the iterates of $\sigma$. It thus follows that $\tilde{\mathbb P}(\tilde B)=0$, and, consequently $\mathbb P(B)=0$. 

\

Let $\nu\notin B$. Then for infinitely many positive integers $k$ {\bf Property $(K,k)$} does not hold.  Pick such $k$. Then $\sigma^i\nu\notin A_{k-i}$ 
for all but at most $K$ indices $i\in\{0,\dots ,k-1\}$.  
Thus,  the assumption of Propostion~\ref{prop:A_k}, (b)  is satisfied for all such indices $k$.
Applying this Proposition we see that  the assumption of Proposition ~\ref{prop:sufficient} is satisfied for $\nu$. 
This allows to  conclude that the Julia set $J_\nu$ is totally disconnected for all $\nu\notin B$. This concludes the proof of Theorem~\ref{thm:main_step}.
\end{proof}

\section{Conclusion. Proof of Theorem A and Theorem B.}
In this section we complete  the proofs of Theorem A and Theorem B.
As shown ih Theorem~\ref{thm:main_step},  it is enough to check that the estimate \eqref{eq: fast_escaping} holds
i.e. the critical point is \emph{typically fast escaping} under the assumptions of both theorems.

First let us note that under the assumptions of Theorem A the estimate \eqref{eq: fast_escaping} was actually proved in \cite{BBR} (see Theorem 2.2 in this paper).

\

Obviously Theorem B implies Theorem A, thus let us focus on the more general setting presented in Theorem B. 
We shall conclude the proof of Theorem B with the following Proposition.

\begin{proposition}\label{prop:integr}
Let $V$ be a bounded open set such that $D(0,\frac{1}{4}) \subset V$ and $V \neq D(0,\frac{1}{4})$. Take $\Omega = V^{\mathbb{N}}$ to be the product space equipped with the product of uniform distributions on $V$, denoted by $\mathbb{P}$. 
There exists a constant $\gamma>0$ such that
$$\mathbb P\left(\{\omega\in\Omega: g_\omega(0)<\frac{G}{2^k}\}\right )<e^{-\gamma k}$$
where $G$ is set as in \eqref{eq:set_G}.
\end{proposition}

\begin{proof}
To prove Proposition ~\ref{prop:integr} we shall use the estimates \eqref{eq:est_green}. We also  need the lemma, which follows the general idea of the  proof of Theorem 2.2 in \cite{BBR}:

\begin{lem}\label{lem:generalpar}
Let $V$ be an open and bounded set,  such that $\mathbb{D}(0,\frac{1}{4}) \subset V\subset \mathbb D(0,R)$  and $V \neq \mathbb{D}(0,\frac{1}{4})$. 
Consider the space $\Omega = V^{\mathbb{N}}$ with the product of uniform distributions on $V$. Then there exists $\gamma>0$ such that for every $z\in\mathbb C$
$$
\mathbb P(k(z, \omega) > k) \leqslant e^{- \gamma k},
$$
where $k(z,\omega)$ is the escape time of $z$  from the disc $\mathbb D_{R_0}$, defined in Definition~ \ref{def:escape}.
\end{lem}

\begin{proof}
Let $c \in V$ be a point such that $|c| > \frac{1}{4}$, say $|c| > \frac{1}{4} + \varepsilon$ for some small $\varepsilon > 0$.

 Let us pick a point $c' \in \mathbb{D}(0,\frac{1}{2})$  (not neccesarily in $V$), 
 such that $|c'| = \frac{1}{2} -  \frac{\varepsilon}{2}$ and $\arg(c') = \frac{\arg(c)}{2}$. 
 In particular, pick $\varepsilon$ small enough so that $\frac{1}{2} -  \frac{\varepsilon}{2} > 0$. 
Observe that for the parabolic map  $f(w) = w^2 + \frac{1}{4}$  we have

\begin{equation}\label{eq:realkwad}
f^n(w)\xrightarrow[n\to\infty]{}\frac{1}{2}
\end{equation}
for every real $w<\frac{1}{2}$.

Consider first $z$ such that  $|z|< \frac{1}{2}$.
We claim that one can choose $N\in\mathbb N$ and the  parameters $c_1,c_2, ..., c_N \in \mathbb{D} (0,\frac{1}{4})$ in a way that $f^N_{\omega}(z) = c'$. 
Indeed,
first note that, since $|z|<\frac{1}{2}$, the set
$$\left \{z^2+c: c\in \mathbb D \left (0,\frac{1}{4} \right )\right \}$$
contains the disc $\{w:|w|< \rho\}$, where $\rho=\frac{1}{4}-|z|^2>0$. So, we can choose $c_0$ such that, putting $w= z^2+c_0$ we have
$|w|<\rho$, and, adjusting $c_0$, we can additionally achieve that the argument of $w$ is 
as we wish.

Using \eqref{eq:realkwad} we  find $N>0$ and real parameters $\tilde  c_1,\dots,\tilde c_{N-1}\in (0,\frac{1}{4})$ such that
$$f^{N-1}_{\tilde c_{N-1},\dots,\tilde c_1}(|w|)=|c'|.$$
Now, choosing appopriate $c_0$, we adjust the argument of $w$ in such a way that
\begin{equation}\label{eq:arg_w}
\left [2^{N-1}\arg(w) =2^{N-1}\arg (f_{c_0}(z)\right ]_{{\rm mod}2\pi}=\arg(c')
\end{equation}
Next, for $n=1,\dots,  N-1$,  we choose $c_n$ in such a way that $|c_n|=\tilde c_n$ and 
$$\arg(c_{n})=\arg\left ((f^{n}_{c_{n-1},\dots c_1,c_0}(z))^2\right )$$
so that
 $$
 \aligned
 |f^{n+1}_{c_n,c_{n-1},\dots, c_1,c_0}(z)|&=|(f^n_{c_{n-1},\dots c_1,c_0}(z))^2+c_n|= |f^{n}_{c_{n-1},\dots c_1,c_0}(z))|^2+|c_n|\\
 &=|f^{n}_{c_{n-1},\dots c_1,c_0}(z))|^2+\tilde c_n,
 \endaligned
 $$
and, in consequence, $|f^N_{c_{N-1},\dots c_0}(z)|=|c'|$ and $\arg(f^N_{c_{N-1},\dots c_0}(z))=\arg(c')$, thus $f^N_{c_{N-1},\dots c_0}(z)=c'$.
Now since $f^N_{c_{N-1},\dots c_0}(z) = c'$ and $\arg [(c')^2]=\arg (c)$,  then, putting  $c_{N}: = c$ we obtain  
$$|f^{N+1}_{c_{N+1},\dots c_0}(z)|=|c'|^2+c >\left (\frac{1}{2}-\frac{\varepsilon}{2}\right )^2+\frac{1}{4}+\varepsilon> \frac{1}{2}.$$

Recall that for $v$ real,  $v > \frac{1}{2}$ we have  $f^n(v)\xrightarrow[n\to\infty]{}\infty$. This means we can pick parameters $c_{N+2}, c_{N+3}, ..., c_{N + N_1-1} \in \mathbb{D} (0,\frac{1}{4})$ for some $N_1$ (again, adjusting the argument appropriately)  in such a way that $|f^{N+N_1}_{c_{N+N_1-1,\dots c_0}}(z)| > R_0+1$. 
For $|z|>\frac{1}{2}$ we obtain the same statement even in a easier way; one only has to repeat the second part of the  reasoning above.
The case of $z$ with $|z|=\frac{1}{2}$ needs a small modification: choosing an appropriate $c_0$ in   $ \mathbb D(0,\frac{1}{4})$, we obtain $|z^2+c_0|<\frac{1}{2}$ and the previously 
described procedure applies.   

So, finally, we checked the following: For every $z\in\mathbb C$, there exists $M=M_z$
and a sequence $c_0,c_1,\dots, c_{M}$, $c_i\in V$, such that 
$$|f^M_{c_{M-1},\dots ,c_0}(z) |>R_0+1.$$

Clearly, the same is true with $c_i$ slightly perturbed, so, if we take $\delta>0$ sufficiently small and put
$$A_z=\mathbb D(c_0,\delta)\times\dots\times \mathbb D(c_{M-1},\delta)\times \mathbb D(0,R)$$
then $\mathbb P(A_z)>0$ and, for all $\omega\in A_z$, $$|f^M_\omega(z)|>R_0+\frac{1}{2}.$$
Since the family 

$$\{{f^M_\omega}_{|\mathbb D_{R_0}},  \omega\in\Omega\}$$
is equicontinuous, we conclude that there exists $U_z\ni z$, an open neighbourhood of $z$ such that for all $v\in U_z$, $\omega\in A_z$, $|f^M_\omega(v)|>R_0$, and because of \eqref{eq:1}, for all $N\ge M$ there holds

$$|f^N_\omega(v)|>R_0.$$

By compactness of $\overline{\mathbb D}_{R_0}$, there exists a finite cover  of $\overline{\mathbb D}_{R_0}$   by a finite collection of the sets $U_{z_i}$.  
Taking $\alpha:=\min_ {z_i}\mathbb P(A_{z_i})$ and $M=\max_{z_i} M_{z_i}$, we can write 

$$
\exists_{M} \exists_{\alpha>0} \forall_z \mathbb{P} (\{ |f^N_{\omega}(z)|>R_0\}) > \alpha.
$$
In other words, putting $S^k (z) = \{\omega \in \Omega : f^k_{\omega} (z) < R_0  \}$,  we know  that for any $z$ we have $\mathbb{P} (S^N (z)) < 1 - \alpha$.

We proceed to estimate $\mathbb P(S^k(z))$ exactly like in \cite{BBR}, using the fact that $\mathbb P$ is the product measure :

$$\mathbb P(S^{k + N} (z)) = \int_{S^k (z)} \mathbb P(S^N (f^k_\omega (z))) d \mathbb P(
\omega)\leqslant (1- \delta) \mathbb P (S^k (z))$$

which applied repeatedly yields the existence of a constant $\gamma > 0$ such that $$\mathbb P(k(z,\omega)>k)=\mathbb P (S^k (z)) \leqslant e^{-k \gamma}.$$ 

\end{proof}

Applying the above result for $z=0$, together with the previously established~\eqref{eq:est_green},
yields the claim, with possibly modified constant $\gamma$.
This ends the proof of Proposition ~\ref{prop:integr}.
\end{proof}

It is important to point out that Lemma \ref{lem:generalpar} is the only part of the proof of the main result that uses the assumption on the parameter space, i.e. that it contains points from outside of the disk $D(0,\frac{1}{4})$. As mentioned before, if $R \leqslant \frac{1}{4}$ then the resulting Julia set is always connected, thus the proof above illustrates exactly the role this assumption fulfills. 

Taking $V$ to be the main cardioid yields the following interesting corollary of Theorem B.

\begin{theorem}\label{thm:cardioid}
Let $\Omega = B^{\mathbb{N}}$ where $B$ is the main cardioid of the Mandelbrot set, and let $\Omega$ be equipped with the product of uniform distributions on $V$. 
Then for almost every sequence $\omega \in \Omega$ the Julia set $J_\omega$ is totally disconnected.
\end{theorem}

\section{Further generalizations.}\label{sec:remarks}

A number of other generalizations can be made by simple adaptations of the proof.
For instance  it can be seen by inspecting the proof of  Lemma ~\ref{lem:generalpar} that  the uniform distribution  does not  play any important role.

\begin{theorem}
Let $R > \frac{1}{4}$, and let $\mu$ be a Borel probability distribution on $\mathbb 
D(0,R)$ such that  ${\rm supp}(\mu)\supset \mathbb D\left (0,\frac{1}{4}\right )$ and  $\mu 
\left (\mathbb{D}(0,R) \setminus (\overline{\mathbb{D}(0,\frac{1}{4})}\right ) > 0$.
 Now consider the product measure of $\mu$ on $\mathbb D(0,R)^{\mathbb{N}}$. The Julia set for a sequence $\{c_n \} \subset \mathbb D(0,R)^{\mathbb{N}}$
 is almost always totally disconnected, with respect to this product measure.
\end{theorem}

The following result comes from \cite{QIU}, Theorem 2.2  but can also be inferred easily from our proof.

\begin{remark*}
For every $c \notin \mathbb{M}$ there exists a neighbourhood $U(c)$ such that $J(c_n)$ is totally disconnected if all $c_n \in U(c)$. 
\end{remark*}
Indeed, in this case it is easy to see that 
$$\inf_{\omega=(c_n),c_n\in U}g_\omega(0)>a>0$$ for some constant $a$, depending on $U$. So, with $K$ sufficiently large, the set $E$ defined in Section~\ref{sec:typically} is just the whole space $\tilde\Omega$. By Lemma ~\ref{lem:*} we conclude that for every $\nu\in \Omega=U(c)^\mathbb N$ and for all $k$
$$\sigma^i\nu\notin A_{k-i}$$ happens for all but all most $K$ indices $i\in\{0,\dots, k-1\}$, which, by Proposition~\ref{prop:A_k} and Proposition~\ref{prop:sufficient} 
immediately implies that every Julia set $J_\omega$ is totally disconnected. 
 
\
Note that another easy adaptation of the proof yields an answer to a question from \cite{BBR} (see Remark 2.5 in \cite{BBR}), whether 
we can choose the parameters randomly, according to the uniform distribution,  from a circle of radius $\delta>1/4$.

Actually, the authors ask in Remark 2.5 in \cite{BBR}  if the set Julia set is almost surely disconnected. Our approach gives much more: 

\begin{proposition}
Let $\Omega = \partial \mathbb D(0,R)^{\mathbb{N}}$ where $R > \frac{1}{4}$ be equipped with the product of uniform distributions on the circle $\partial\mathbb D(0,R)$. 
Then for almost every $\omega \in \Omega$ the Julia set $J_\omega$ is totally disconnected.
\end{proposition}

To repeat our proof in the above case we need the following version of Lemma \ref{lem:generalpar}.

\begin{lem}
Let $K = \partial \mathbb D(0,R)$ where $R > \frac{1}{4}$. Consider the space $\Omega = K^{\mathbb{N}}$ with the product of uniform distributions on $K$. 
Then there exists  $\gamma>0$ such that for all $z\in\mathbb C$,
$$
\mathbb P(k(z, \omega) > k) \leqslant e^{- \gamma k},
$$
where $k(z,\omega)$ is the value defined in \eqref{eq:escape}.
\end{lem}

\begin{proof}
Take an arbitrary point $z \in \mathbb{C}$, let $c_1, c_2, c_3, ... , c_N \in K$ be a sequence of $N$ parameters such that for all $n\le N$
$$|f^n_{\omega}(z)| = |f^{n-1}_{\omega}(z)^2 + c_n| = |f^{n-1}_{\omega}(z)|^2 + |c_n|. $$ 
Recall that for iterations on the real line, with $f(x) = x^2 + R$ and $R > \frac{1}{4}$, 
we have  for all $x$ $$\lim_{n \rightarrow \infty} f^n(x) = \infty .$$ Since $|c_n| = R > 
\frac{1}{4}$ by our choice of the numbers $c_1, ..., c_N$, for a large enough $N$, we will have $|
f^N_{\omega} (z)| > R_0$. By  continuity and compactness arguments, used exactly  as in the proof of Lemma~\ref{lem:generalpar}, 
we see that one can show something more, that is

$$
\exists_N \exists_{\delta} \forall_z \mathbb{P} (\{ \omega \in \Omega :  |f^N_{\omega} (z)| > R_0 \}) > \delta.
$$
We finish the proof in exactly the same way as the proof of Lemma~\ref{lem:generalpar}.

\end{proof}

\

\bibliographystyle{amsplain}

\end{document}